\documentclass{amsart}

\usepackage{amssymb}
\usepackage{amsmath}
\usepackage{amscd}
\usepackage{ascmac}
\usepackage{graphicx}
\usepackage{pinlabel}

\theoremstyle{plain}
\newtheorem{thm}{Theorem}[section]
\newtheorem{cor}[thm]{Corollary}
\newtheorem{lem}[thm]{Lemma}
\newtheorem{prop}[thm]{Proposition}

\theoremstyle{definition}
\newtheorem{defn}[thm]{Definition}
\newtheorem{exmp}[thm]{Example}

\theoremstyle{remark}
\newtheorem{rem}[thm]{Remark}

\begin{document}

\title{Slice depth of 2-knots}
\author{Ayaka Ise}
\address{}
\email{}
\date{}

\maketitle

\begin{abstract}

We introduce the notion of slice depth of a $2$-knot $K$, which is the minimal integer $n$ such that 
$K$ is $n$-slice. We give an upper bound for the slice depth of the $n$-twist spin 
of a classical knot which belongs to several specific classes, namely, 
certain types of $2$-bridge knots, pretzel knots, ribbon knots of $1$-fusion, and knots with given 
unknotting number. 
\end{abstract}


\section{Introduction}


A $2$-knot is a smoothly embedded $2$-sphere in $S^4$. 
Two $2$-knots $K_0$ and $K_1$ are said to be {\it concordant} if there exists a proper embedding 
$\varphi:S^2\times [0,1]\rightarrow S^4\times [0,1]$ such that $\varphi(S^2\times \{0\})=K_0\times\{0\}$ and 
$\varphi(S^2\times\{1\})=K_1\times\{1\}$, where $\varphi$ is called a {\it concordance} from $K_0$ to $K_1$. 
For a non-negative integer $n$, $K_0$ and $K_1$ are said to be $n$-{\it concordant} if there exists 
a concordance $\varphi$ from $K_0$ to $K_1$ such that the maximal genus of components of 
the closed orientable surface $\varphi(S^2\times [0,1])\cap (S^4\times \{t\})$ for all generic $t\in [0,1]$ 
is at most $n$ (See Melvin \cite{Melvin}). 
Kervaire \cite{Kervaire} proved that every $2$-knot $K$ is {\it slice}, that is, concordant to a trivial $2$-knot. 
This implies that every $2$-knot $K$ is $n$-{\it slice}, that is, 
$n$-concordant to a trivial $2$-knot, for sufficiently large $n$. 
This paper aims to determine the minimal integer $n$ for which a given $2$-knot is $n$-slice. 

\begin{defn}
The {\it slice depth} $\operatorname{sd}(K)$ of a $2$-knot $K$ is the minimal integer $n$ such that $K$ is $n$-slice. 
\end{defn}

Note that a $2$-knot $K$ is $n$-slice if and only if $\operatorname{sd}(K)$ is less than or equal to $n$. 
In particular, $K$ is $0$-slice if and only if $\operatorname{sd}(K)$ is equal to $0$. 

It was not known until recently whether there exists a $2$-knot which is not $0$-slice 
(See Problem 1.105 on the Kirby problem list \cite{Kirby}). 
Sunukjian \cite{Sunukjian} proved that there are infinitely many distinct $0$-concordance classes of $2$-knots, 
confirming the existence of $2$-knots which are not $0$-slice. 
Dai and Miller \cite{DM} proved that the monoid of $0$-concordance classes of $2$-knots 
under the connected sum operation is not finitely generated. 
Joseph \cite{Joseph} showed that the Alexander ideal induces a homomorphism from the $0$-concordance 
monoid of $2$-knots to the ideal class monoid of the Laurent polynomial ring over $\mathbb{Z}$. 
From these results, we have many examples of $2$-knots with positive slice depth. 

In this paper, we investigate methods for estimating the slice depth from above. 
We will give an upper bound for the slice depth of the $n$-twist spin $\tau^n(k)$ of a classical knot $k$ 
which belongs to several specific classes. For every positive integer $n$, we will show 

$\bullet$ 
$\operatorname{sd}(\tau^n(k))\leq 1$ for certain types of $2$-bridge knots $k$ (Theorem \ref{THM2-bridge}); 

$\bullet$ 
$\operatorname{sd}(\tau^n(k))\leq 2$ for certain types of pretzel knots $k$ (Theorem \ref{THMPretzel}); 

$\bullet$ 
$\operatorname{sd}(\tau^n(k))\leq 2$ for certain types of ribbon knots of $1$-fusion $k$ (Theorem \ref{THMribbon}); 

$\bullet$ 
$\operatorname{sd}(\tau^n(k))\leq u$ for certain types of knots $k$ with unknotting number $u$ (Theorem \ref{THMunknotting}). 

\noindent
Combining Theorem 4.1 below with Corollary 1.2 of \cite{DM}, 
we conclude that $\operatorname{sd}(\tau^n(k))= 1$ for certain types of $2$-bridge knots $k$. 
In order to obtain upper bounds exhibited above, 
we will adapt Satoh's argument \cite{Sato} for estimating unknotting numbers of twist spun knots. 

This paper is organized as follows. 
In Section 2, we review a lemma of Satoh \cite{Sato} on isotopies of 
$1$-handles attached to twist spun knots. 
In Section 3, we visualize belt spheres of $1$-handles attached to twist spun knots, 
and describe their changes under isotopies of $1$-handles. 
In Section 4, we prove our main theorems on upper bounds for twist spun knots. 
We end this paper by comparing slice depth with unknotting number for $2$-knots in Section 5.


\section{Isotopies of twist spun knots}


In this section, we review a definition of twist spun knots and a lemma of Satoh \cite{Sato} 
on isotopies of them. 

Let $\mathbb{R}^3_+=\{(x,y,z)\in\mathbb{R}^3\, |\, z\geq 0\}$ be the closed upper half-space 
and $I=[0,1]$ the unit interval. 
A {\it tangle} $T$ in $\mathbb{R}^3_+$ is a union of an arc and circles properly embedded in $\mathbb{R}^3_+$. 
We assume that the knotting part of $T$ is included in the ball 
$B=\{(x,y,z)\in\mathbb{R}^3\, |\, x^2+y^2+(z-2)^2\leq 1\}$ 
and $T-\mathrm{Int}\, B$ is a pair of trivial arcs 
connecting $(0,\pm 3,0)$ to $(0, \pm 1, 2)$ in $\mathbb{R}^3_+$. 
For a non-negative integer $n$, let $f_t^n:\mathbb{R}^3_+\rightarrow\mathbb{R}^3_+\, (t\in I)$ be 
an ambient isotopy such that each $f_t^n$ fixes $T-\mathrm{Int}\, B$ pointwise and 
$\{f_t^n\}_{t\in I}$ rotates the ball $B$ $n$ times around the axis $\{(0,y,2)\in\mathbb{R}^3\, |\, y\in \mathbb{R}\}$. 

The quotient space of the equivalence relation $\sim$ on $\mathbb{R}^3_+\times I$ generated by 
$(x,y,0,t)$ $\sim (x,y,0,s)$ for every $(x,y)\in\mathbb{R}^2$ and $t,s\in I$ and 
$(x,y,z,0)\sim (x,y,z,1)$ for every $(x,y,z)\in\mathbb{R}^3_+$ can be identified with $\mathbb{R}^4$. 
We denote the image of $\{(f_t^n(p),t)\in\mathbb{R}^3_+\times I\, |\, p\in T, t\in I\}$ 
under the natural surjection $\pi:\mathbb{R}^3_+\times I\rightarrow \mathbb{R}^4$ by $F^n(T)$, 
and that of $\mathbb{R}^3_+\times \{t\}$ by $\mathbb{R}^3_+[t]$. 
The $2$-knot $\tau^n(k)=F^n(T)$ in $\mathbb{R}^4 \, (\subset \mathbb{R}^4\cup\{\infty\}=S^4)$ 
is called the $n$-{\it twist spin} of a classical knot $k$ if $T$ has no circles and $k$ is the closure of $T$. 
The $0$-twist spin $\tau^0(k)$ of $k$ is often called the {\it spin} of $k$. 

A {\it band} $b$ for $T$ is an embedding of $I\times I$ into $B$ such that $T\cap b(I\times I)=b(I\times\partial I)$. 
The image $b(I\times I)$ is often denoted by $b$ for short. 
The $1$-handle $h$ associated with $b$ is defined as the image of 
$\{(f_t^n(p),t)\in\mathbb{R}^3_+\times I \, |\, p\in b(I\times I),\, t\in [0,1/2]\}$ under $\pi$. We denote the image of 
$\{(f_t^n(p),t)\in\mathbb{R}^3_+\times I \, |\, p\in b((0,1)\times I),\, t\in (0,1/2)\}$ under $\pi$ by $h^{\circ}$. 
Let $\mathcal{B}=b_1\cup\cdots\cup b_m$ be a disjoint union of $m$ bands for $T$ such that 
$T\cup\mathcal{B}$ is connected, and 
$\mathcal{H}=h_1\cup\cdots\cup h_m$ the union of $1$-handles associated with $b_1,\ldots ,b_m$. 
The surface $F^n(T,\mathcal{B})$ embedded in $\mathbb{R}^4$ is defined as 
$(F^n(T)\cup\mathcal{H})-(h_1^{\circ}\cup\cdots\cup h_m^{\circ})$, 
and it is called the surface-knot obtained from $F^n(T)$ by {\it surgery along $1$-handles} $h_1,\ldots ,h_m$. 
Note that an embedded cobordism between $F^n(T)$ and $F^n(T,\mathcal{B})$ is constructed from 
$F^n(T)\cup\mathcal{H}$. 

\begin{lem}[Satoh \cite{Sato}, Lemma 4]
Let $(T,\mathcal{B})$, $(T',\mathcal{B}')$ be pairs of a tangle without circles and a disjoint union of 
$m$ bands for the tangle such that the union of them is connected. 
If $(T',\mathcal{B}')$ is obtained from $(T,\mathcal{B})$ by a finite sequence of the operations \textup{(1)--(6)}
depicted in Figure 1, then $F^n(T',\mathcal{B}')$ is ambient isotopic to $F^n(T,\mathcal{B})$. 
\end{lem}
If $(T',\mathcal{B}')$ is obtained from $(T,\mathcal{B})$ by a finite sequence of the operations \textup{(1)--(6)}, we say that $(T',\mathcal{B}')$ is {\it equivalent} to $(T,\mathcal{B})$.

\begin{figure}[tbh]
\centering
\includegraphics[width=12cm]{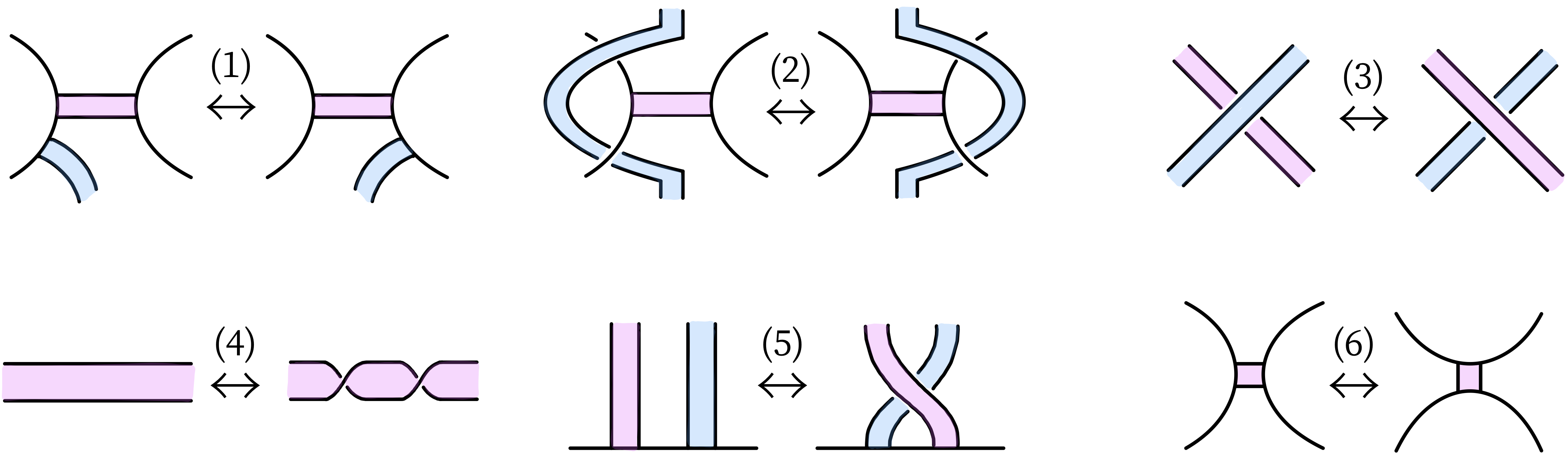}
\caption{Equivalence operations.}
\label{}
\end{figure}


\section{Tracking Belt Spheres of 1-Handles}
\label{}


The purpose of this section is to construct a concordance between a twist spun knot $F^n(T)$ and a trivial $2$-knot.
To this end, we perform 1-handle surgeries on $F^n(T)$ to obtain $F^n(T,B)$. If all 1-handles can be cancelled by corresponding 2-handles, then we obtain the trivial 2-knot. In this way, the intended concordance is obtained, and the number of attached 1-handles provides an upper bound for the slice depth (Proposition~\ref{sliceprop}).
The essential difficulty is whether each 1-handle can be cancelled by attaching a corresponding 2-handle, that is, whether it is possible to attach a 2-handle whose attaching sphere intersects the belt sphere of the corresponding 1-handle transversely in a single point. In particular, we investigate the possibility of attaching 2-handles in this manner by applying the operations (1)--(6) to $(T,B)$ and tracking the changes in the belt spheres of the 1-handles of $F^n(T,B)$.

A significant change in the belt sphere of the 1-handle occurs when the operation (6) is performed. In this operation, the belt sphere, which originally lies in $\bigcup_{t\in[0, 1/2]}\mathbb{R}_+^3[t]$, is shifted to $\bigcup_{t\in[1/2, 1]}\mathbb{R}_+^3[t]$. The upper and lower rows of Figure \ref{belt sphere(6)} present motion pictures showing a portion of $F^n(T,\mathcal{B})$. Each sequence was extracted so as to include the changes resulting from the application of the operation (6).
In the diagram before applying the operation (6), the belt sphere of the 1-handle appears as a blue vertical line at $t=0$, as two points for $0<t<1/2$, and again as a blue vertical line at $t=1/2$.
After applying the operation (6), the belt sphere appears as a blue vertical line at $t=1/2$, as two points for $1/2<t<1$, and finally as a blue vertical line at $t=1$.

\begin{figure}[tbh]
\centering
\includegraphics[width=9cm]{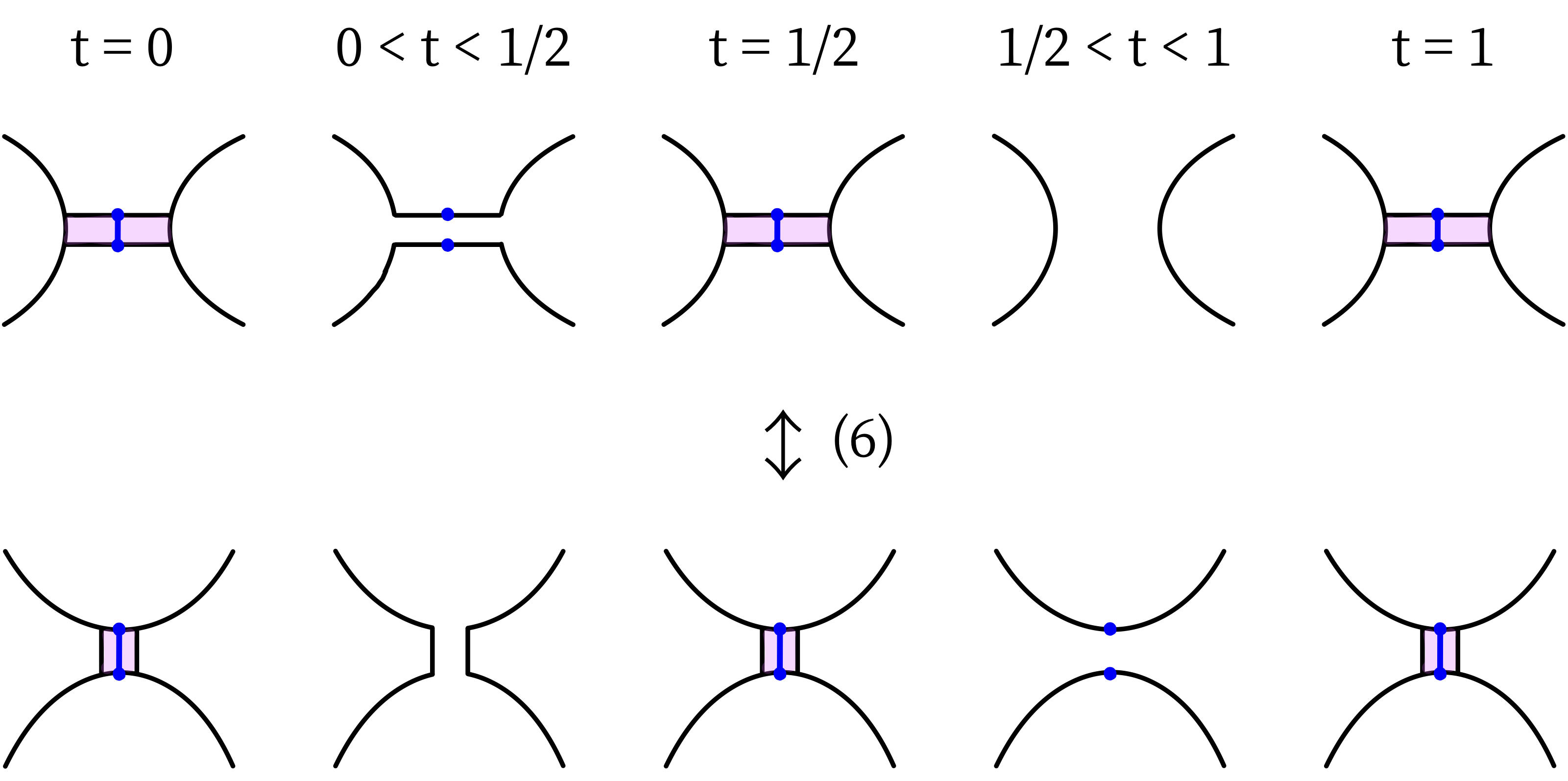}
\caption{Belt sphere change by the operation (6).}
\label{belt sphere(6)}
\end{figure}

We now consider applying other operations after performing the operation (6). The operations (2) and (3) do not cause any problem. As for the operation (4), some care is needed: The upper row of Figure \ref{belt sphere(4)} shows, for each $t\in[0,1]$ in the motion picture in the lower row of Figure \ref{belt sphere(6)}, the result of transferring a full twist from $(T, \mathcal{B})$ to the band. When eliminating a full twist by the operation (4), the entire handle undergoes a twisting motion, which causes the belt sphere to wrap around the $1$-handle. This change is shown as a motion picture in the lower row of Figure \ref{belt sphere(4)}.

\begin{figure}[tbh]
\centering
\includegraphics[width=9cm]{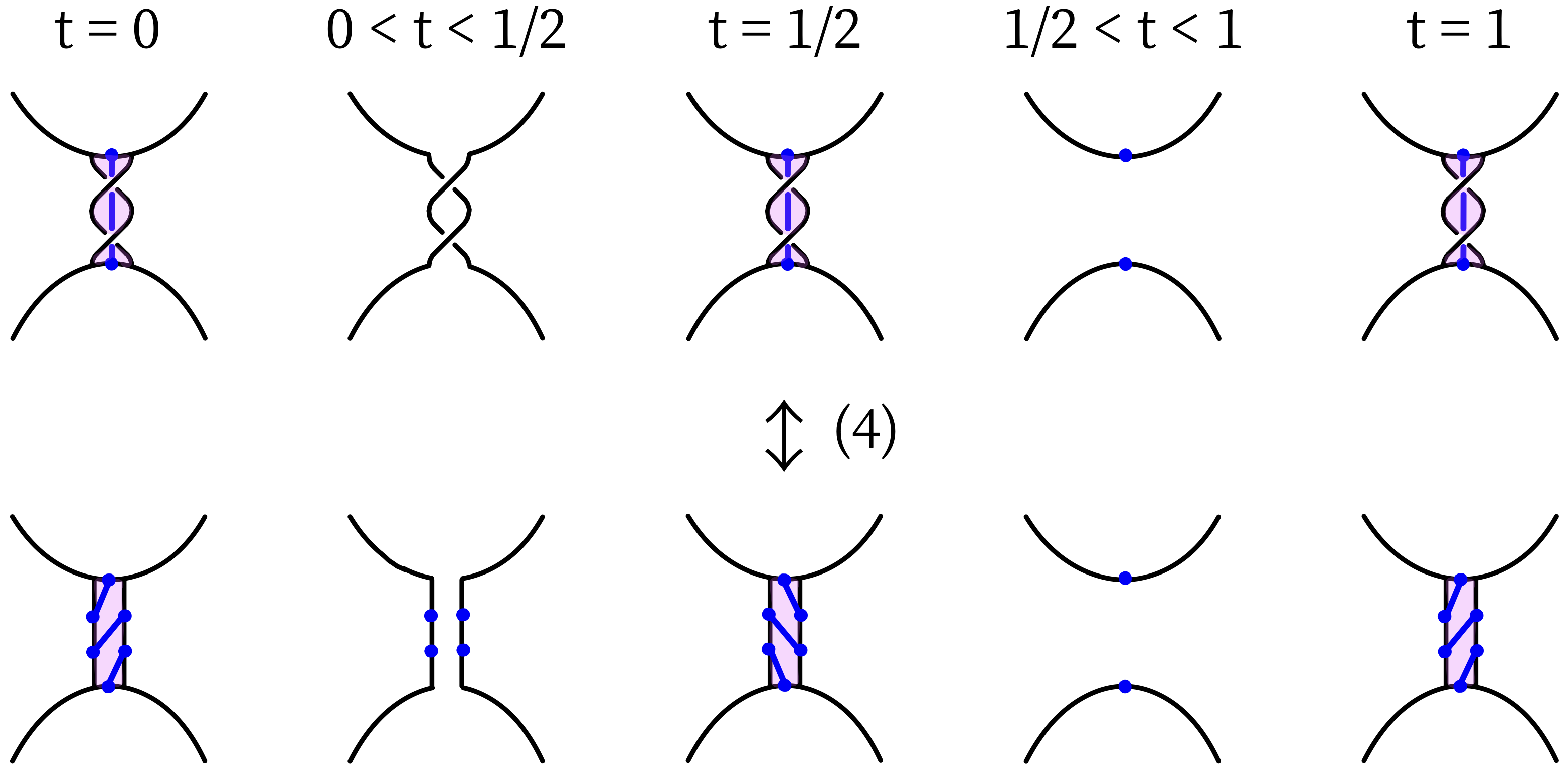}
\caption{Belt sphere change by the operation (4) after (6).}
\label{belt sphere(4)}
\end{figure}

Next, we consider the operations (1) and (5), which involve sliding handles. In these cases, we must take care to ensure that the belt spheres of the interacting handles do not intersect.

In the operation (1), let $h_1$ be the handle being slid and $h_2$ the handle sliding it. 
In the course of the operation, we adjust the belt spheres of $h_1$ and $h_2$ so that they do not intersect as follows: $h_1$ appears for $0\leq t\leq 1/2$, and $h_2$ appears for $1/8\leq t\leq 3/8$. In Figure \ref{belt sphere(1)}, $b_1$ and $b_2$ are the bands corresponding to  $h_1$ and $h_2$ respectively.
In the top row of Figure~\ref{belt sphere(1)}, the belt sphere of $h_i$ ($i=1,2$) appear in the motion picture at each $t\in[0,1]$ as either $b_i(I\times \{1/2\})$ or $b_i(\partial I\times \{1/2\})$.
The middle row of Figure \ref{belt sphere(1)} shows the result of applying the operation (6) to $h_2$ in the top row.
Let $p$ be a point on the subarc of $T$ extending from $b_1(1,1)$, chosen outside the region to which $h_2$ moves.
Denote by $A$ the subarc of $T$ connecting $b_1(1,1)$ and $p$.
We construct the belt sphere of $h_1$ after performing the operation (1) as follows:
\begin{align*}
    &\text{For } t=0,1/2: b_1(I\times\{1/2\})\cup b_1(\{1\}\times[1/2,1])\cup A.\\
    &\text{For } 0<t<1/2: \{b_1(0,1/2)\}\cup\{p\}.
\end{align*}
We construct the belt sphere of $h_2$ after performing the operation (1) as follows.  
For the original position of $b_2$, set $b_2(1/2,0)=q$.
Let $b'_2$ denote the band corresponding to $h_2$ after the operation. Then:
\begin{align*}
    &\text{For } t=1/8,3/8: b'_2(\{1/2\}\times I)\cup\{\text{the arc traced by }b_2(1/2,0)\}.\\
    &\text{For } 0\leq t<1,3/8<t\leq 1: \{q\}\cup \{b'_2(1/2,1)\}.
\end{align*}
This is illustrated in the bottom row of Figure \ref{belt sphere(1)}.
\begin{figure}[tbh]
\centering
\includegraphics[width=12.5cm]{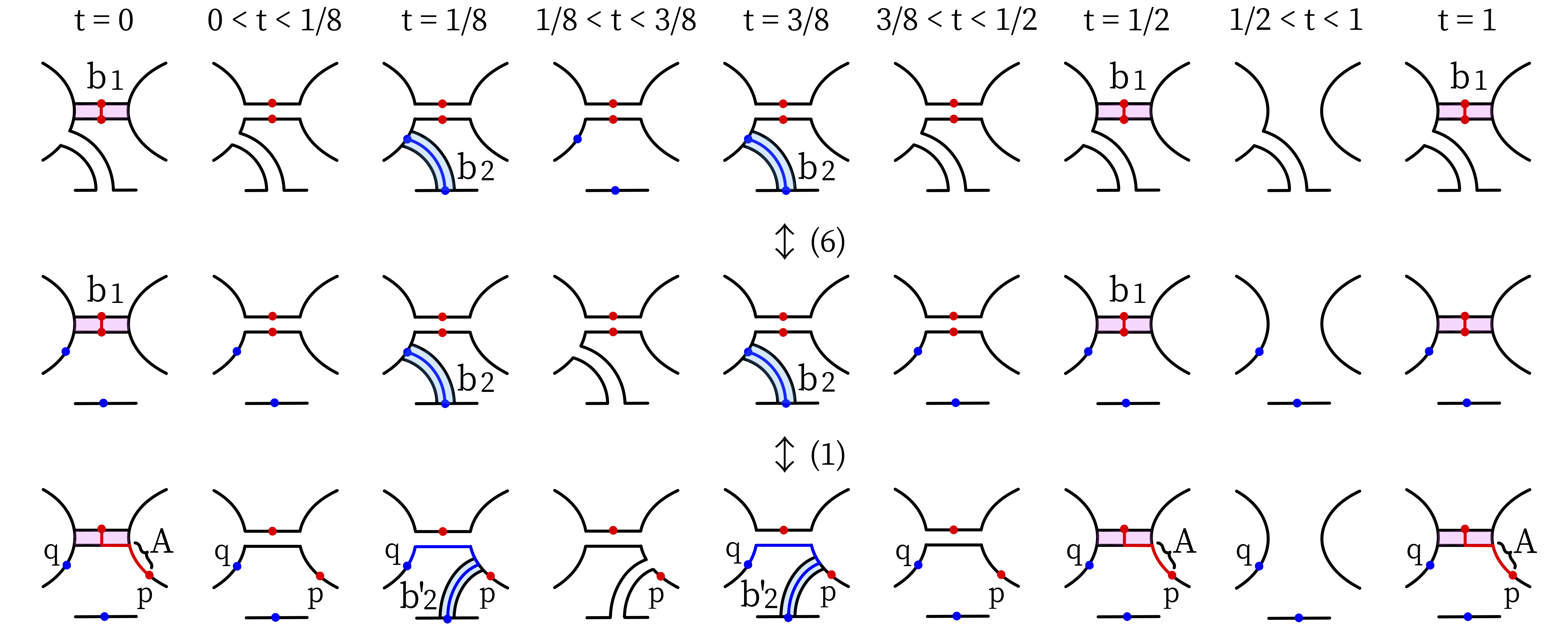}
\caption{Belt sphere change by the operation (1) after (6).}
\label{belt sphere(1)}
\end{figure}

In the course of the operation (5), we adjust the belt spheres of $h_1$ and $h_2$ so that they do not intersect as follows: $h_1$ appears for $0\leq t\leq 1/8$, and $h_2$ appears for $3/8\leq t\leq 1/2$. In Figure \ref{belt sphere(5)}, $b_1$ and $b_2$ are the bands corresponding to  $h_1$ and $h_2$ respectively.
The middle row of Figure \ref{belt sphere(5)} shows the result of applying the operation (6) to $h_1$ and $h_2$ in the top row.
For $i=1,2$, set $b_i(1/2,1)= p_i$ for the original position of $b_i$.
Let $b'_i$ denote the band corresponding to $h_i$ after the operation.
We construct the belt sphere of $h_1$ after performing the operation (5) as follows:
\begin{align*}
    &\text{For } t=0,1/8: b'_1(\{1/2\} \times I )\cup \{\text{the arc traced by }b_1(1/2, 1)\}.\\
    &\text{For } 1/8<t<1: \{b'_1(1/2,0)\}\cup\{p_1\}.
\end{align*}
We construct the belt sphere of $h_2$ after performing the operation (5) as follows:
\begin{align*}
    &\text{For } t=3/8,1/2: b'_2(\{1/2\} \times I )\cup \{\text{the arc traced by }b_2(1/2,1)\}.\\
    &\text{For } 0\leq t<3/8, 1/2<t\leq 1 : \{b'_2(1/2,0)\}\cup\{p_2\}.
\end{align*}
This is illustrated in the bottom row of Figure \ref{belt sphere(5)}.

\begin{figure}[tbh]
\centering
\includegraphics[width=12.5cm]{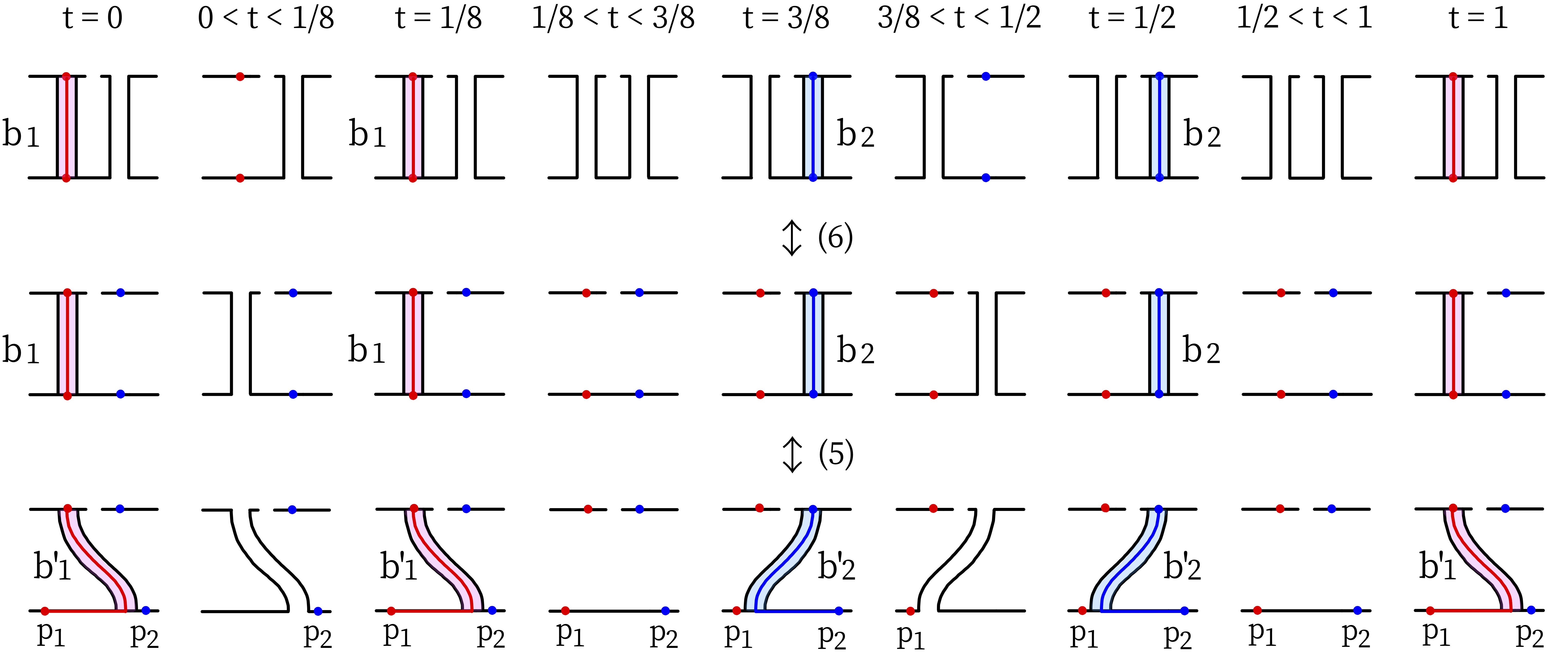}
\caption{Belt sphere change by the operation (5) after (6).}
\label{belt sphere(5)}
\end{figure}

\begin{rem}
As shown in Figure \ref{twisting and rolling}, a full twist on a band can be rewritten so that the band appears to roll along a circle. When two full twists are applied, we can use the operation (3) to switch the over/under crossings of one of the rollings, thereby cancelling the two full twists (or rollings) as a pair.

\begin{figure}[tbh]
\centering
\includegraphics[width=5cm]{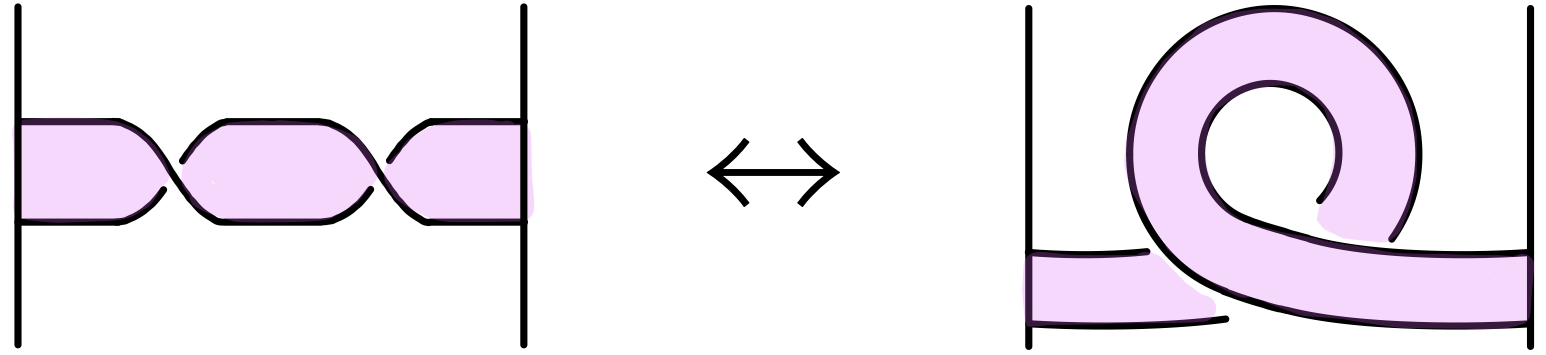}
\caption{Full twist on a band rewritten as a rolling position.}
\label{twisting and rolling}
\end{figure}
\end{rem}

We now describe the setting for Proposition~\ref{sliceprop}.
Let $T_0$ be a trivial tangle, $\mathcal{B}_0=b_1\cup\dots\cup b_m$ be a union of 
$m$ disjoint bands attached to $T_0$, and $S^*_1,\dots,S^*_m$ be $m$ disjoint 
circles on $F^n(T_0,\mathcal{B}_0)$. 
Let a circle $S_1$ be formed by $b_1(\{0\}\times I)$ and the subarc $A_1$ of $T_0$ 
connecting $b_1(0,0)$ and $b_1(0,1)$ as in Figure \ref{delete}.  
\begin{figure}[tbh]
\centering
\includegraphics[width=4cm]{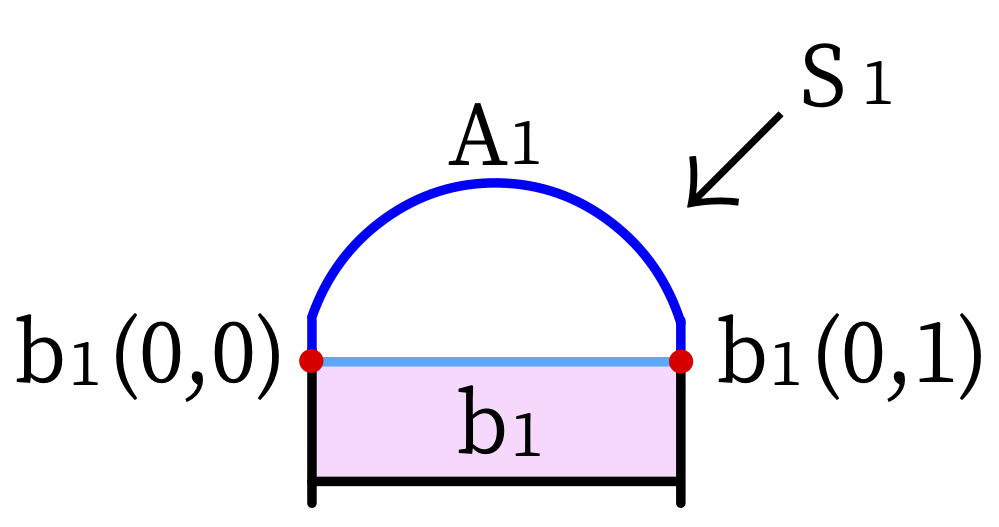}
\caption{$A_1$ (blue) and $b_1({0}\times I)$ (light blue) form the circle $S_1$, connecting the red points $b_1(0,0)$ and $b_1(0,1)$.}
\label{delete}
\end{figure}

\noindent We require the following conditions for $b_1$ and $T_0\cup\mathcal{B}_0$:
\begin{enumerate}
\item No other bands pass through $S_1$.
\item $S^*_1$ appears at a single point on $b_1(\{0\}\times I)$.
\item Removing $b_1([0,1)\times I) \cup A_1$ from $T_0\cup\mathcal{B}_0$ 
results again in the union of a trivial tangle and the remaining bands 
$b_2\cup\dots\cup b_m$.
\end{enumerate}
For each $2\leq i\leq m$, let $A_i$ be a portion of 
$T_0\cup\mathcal{B}_0-\Big(\bigcup_{j=1}^{i-1}b_j([0,1)\times I) \cup A_j\Big)$
that connects $b_i(0,0)$ and $b_i(0,1)$ and Let $S_i$ be a circle formed by $b_i(\{0\}\times I)$ and the $A_i$.
We assume similarly that for each $2\leq i\leq m$, the pair of $b_i$ and 
\begin{equation*}
    T_0\cup\mathcal{B}_0-\Bigg(\bigcup_{j=1}^{i}b_j([0,1)\times I) \cup A_j\Bigg)
\end{equation*}
satisfy the conditions. We further assume that $T_0\cup\mathcal{B}_0-\Big(\bigcup_{j=1}^{m}b_j([0,1)\times I) \cup A_j\Big)$ is the trivial tangle without any bands. The two motion pictures in Figure \ref{(T_0,B_0)} give examples of condition (2), where $S^*_i$ appear as a blue line or a point.

\begin{figure}[tbh]
\centering
\includegraphics[width=11cm]{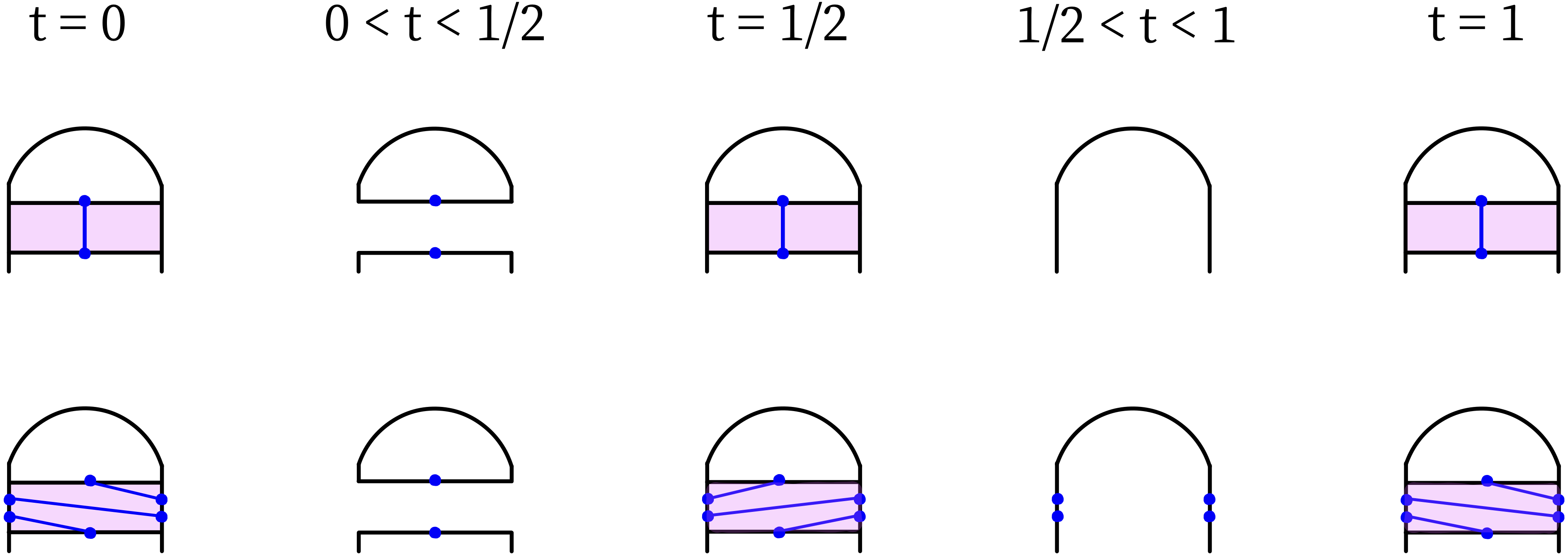}
\caption{Examples of condition (2).}
\label{(T_0,B_0)}
\end{figure}

\begin{prop}\label{sliceprop}
Let $k$ be a classical knot, $T$ be a tangle representing $k$, and $\mathcal{B}$ be a union of $m$ disjoint bands attached to $T$.
Let $(T_0, \mathcal{B}_0)$, $S_1,\dots,S_m$ and $S^*_1,\dots,S^*_m$ be as described in the above setting.
If $(T,B)$ is equivalent to $(T_0,\mathcal{B}_0)$ and the belt spheres of the 
1-handles of $F^n(T,B)$ appear as $S^*_1,\dots,S^*_m$ in $F^n(T_0,\mathcal{B}_0)$, then the $n$-twist spin $\tau^n(k)$ is $m$-slice.
\end{prop}

\begin{proof}
    If $(T, \mathcal{B})$ is equivalent to $(T_0, \mathcal{B}_0)$, then $F^n(T, \mathcal{B})$ is ambient isotopic to $F^n(T_0, \mathcal{B}_0)$ by Lemma 2.1.
    Under this isotopy, the circles corresponding to the belt spheres of the 1-handles of $F^n(T, \mathcal{B})$ are carried to those of $F^n(T_0, \mathcal{B}_0)$ as $S^*_1,\dots,S^*_m$.
    We now take the 1-handle $h_i^1$ of $F^n(T, \mathcal{B})$ whose belt sphere is the circle obtained by pulling back $S^*_i$ to $F^n(T, \mathcal{B})$ via the inverse of the isotopy. 
    Let $D_i$ be the disk in $(T_0, \mathcal{B}_0)$ whose boundary is $S_i$. Let $\tilde{D}_i$ be the image of $D_i$ in $(T, \mathcal{B})$ under the inverse of the isotopy. Let $h^2_i$ be the 2-handle of $F^n(T,\mathcal{B})$ whose core is  $\tilde{D}_i$. Then, since the belt sphere of $h^1_1$ intersects $h^2_1$ transversely in a single point, the pair $(h_1^1,h_1^2)$  forms a cancelling pair.
    After cancelling $h^1_1$ and $h^2_1$ we proceed similarly by attaching $h^2_i$ to each
    $h^1_i$ for each $2\leq i \leq m$, cancelling them one by one. 
    As a result, all the 1-handles of $F^n(T, \mathcal{B})$ can be cancelled, and we obtain a trivial 2-knot in $S^4$.
Thus, there exists a cobordism $X$ in $S^4\times I$  whose boundary is $F^n(T)\sqcup \text{(trivial 2-knot)}$. By handle decomposition, we have
  \begin{equation*}
      X=(S^2\times I)\cup h^1_1\cup\dots\cup h^1_m\cup h^2_1\cup\dots\cup h^2_m
  \end{equation*}
Since all the handles can be cancelled, $X$ is diffeomorphic to $S^2\times I$. Therefore, this cobordism $X$ is a concordance between $F^n(T)$ and a trivial $2$-knot.
The maximal genus of components of the closed orientable surface $X\cap(S^4\times \{t\})$ for all generic $t\in[0,1]$ is equal to $m$, and thus $X$ is a $m$-concordance.
Hence, $F^n(T)=\tau^n(k)$ is $m$-slice.
\end{proof}


\section{Main Results}
\label{Main Results}


In this section, we provide upper bounds for the slice depth of the $n$-twist spin of certain classes of classical knots.

\subsection{2-bridge knots}
A $2$-bridge knot is a classical knot in $\mathbb{R}^3$ that can be arranged to have exactly two local maxima with respect to the $z$-axis. Every 2-bridge knot can be represented by a sequence of nonzero integers $(a_1,\dots a_m)$ and it is denoted by $C(a_1, \dots, a_m)$. It is possible to associate a continued fraction with the integer sequence  $(a_1,\dots a_m)$ corresponding to a 2-bridge knot.
Since any $2$-bridge knot admits a continued fraction expansion whose coefficients are all even, we adopt such an expansion throughout this paper and use the corresponding diagram.
For details on 2-bridge knots, see Kawauchi \cite{Kawauchi}.

\begin{thm}\label{THM2-bridge}
    For any positive integer $n$, the $n$-twist spin of the 2-bridge knot 
    $C(a_1, \dots, a_m)$ is 1-slice if the following condition is satisfied.
    Let $(b_1,\dots,b_m)$ be the sequence obtained by reducing each $a_i$ modulo $4$.
    From this sequence, omit every term with $b_i=0$; if $b_i=2$, replace it with $O$ when $i$ is odd, and with $E$ when $i$ is even. This process yields a word $w_k$ of length $k$ $(0 \leq k \leq m)$ in the letters $O$ and $E$.
    If $w_k$ reduces to either the empty word or the single letter $O$ by repeatedly deleting the initial $E$ or adjacent pairs $OO$ or $EE$, then the $n$-twist spin is 1-slice.
    \end{thm}

\begin{proof}
    Since all $a_1, \dots, a_m$ are even, $m$ must be even; otherwise, $C(a_1, \dots, a_m)$ represents a link rather than a knot. Figure \ref{2-bridge knot} represents $C(a_1, \dots, a_m)$ and its corresponding tangle with an appropriately attached single band. Each box with an odd index $i$ indicates $a_i$ positive half-twists, while each box with an even index $i$ represents $a_i$ negative half-twists.

    \begin{figure}[tbh]
\centering
\includegraphics[width=8cm]{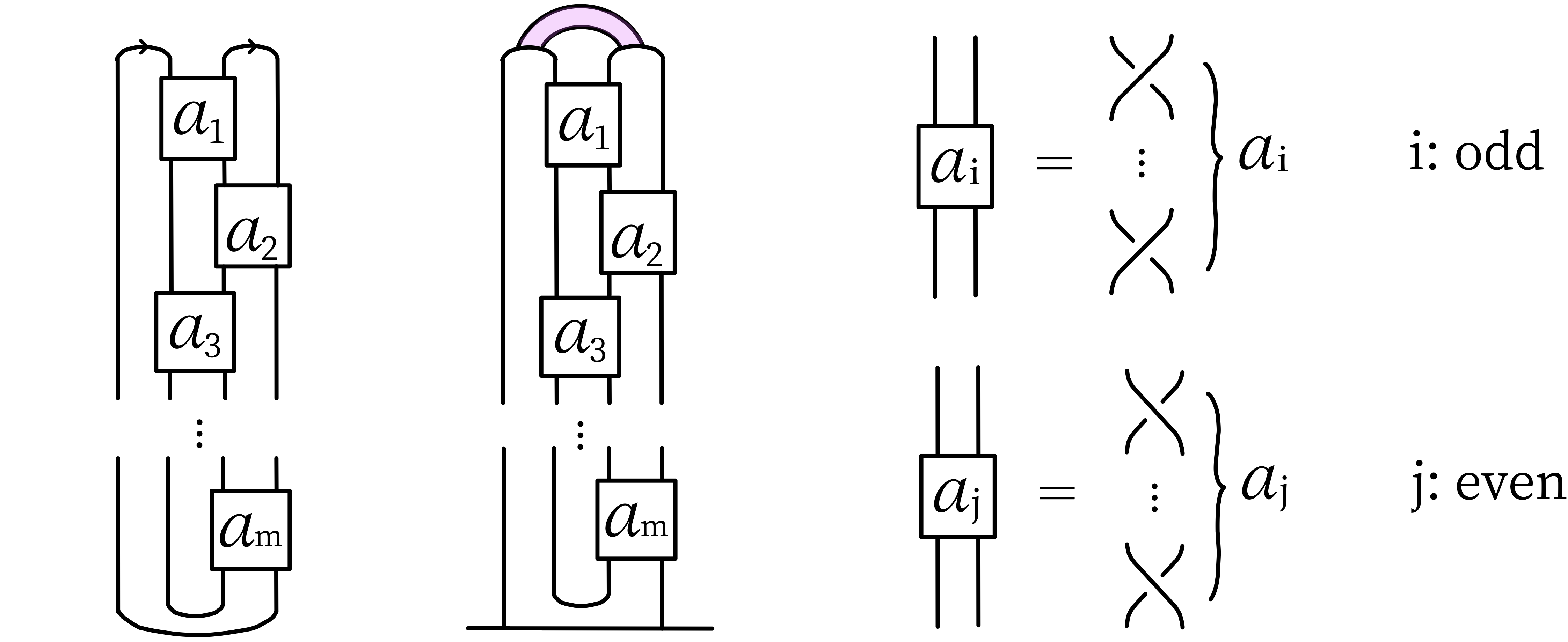}
\caption{$C(a_1, \dots, a_m)$ and the corresponding banded tangle.}
\label{2-bridge knot}
\end{figure}

    We first consider the cases of $C(a_1, a_2)$ when $(a_1,a_2)\equiv(0,0),(0,2),(2,0),(2,2) \pmod4$.

    (i) When $(a_1,a_2)\equiv(0,0) \pmod4$: 
     Applying the operation (6) to the band makes it vertical and transfers $a_1$ twists onto it. Since $a_1\equiv0$ $\pmod4$, these twists can be eliminated by the operation (3), as explained in Remark 3.1. Applying the operation (6) once more makes the band horizontal, so that $a_2$ twists are transferred onto it, which can again be eliminated by the operation (3) since $a_2\equiv0\pmod4$. Consequently, the tangle becomes trivial, the band is attached without twists, and the belt sphere is vertical.

    (ii) When $(a_1,a_2)\equiv(0,2) \pmod4$: 
    Applying the operation (6) to the band makes it vertical and transfers $a_1$ twists onto it. Since $a_1\equiv0\pmod4$, the twists can be eliminated by the operation (3). Performing the operation (6) again makes the band horizontal, and $a_2$ twists are transferred onto it. Since $a_2\equiv2\pmod4$, eliminating the twists by the operation (3) leaves one full twist. This twist can be removed by the operation (4). Hence the tangle becomes trivial, the band remains untwisted, and the belt sphere is vertical.

    Figure \ref{2-bridge transform1} shows the motion pictures before and after the deformations (i) and (ii).

 \begin{figure}[tbh]
\centering
\includegraphics[width=10cm]{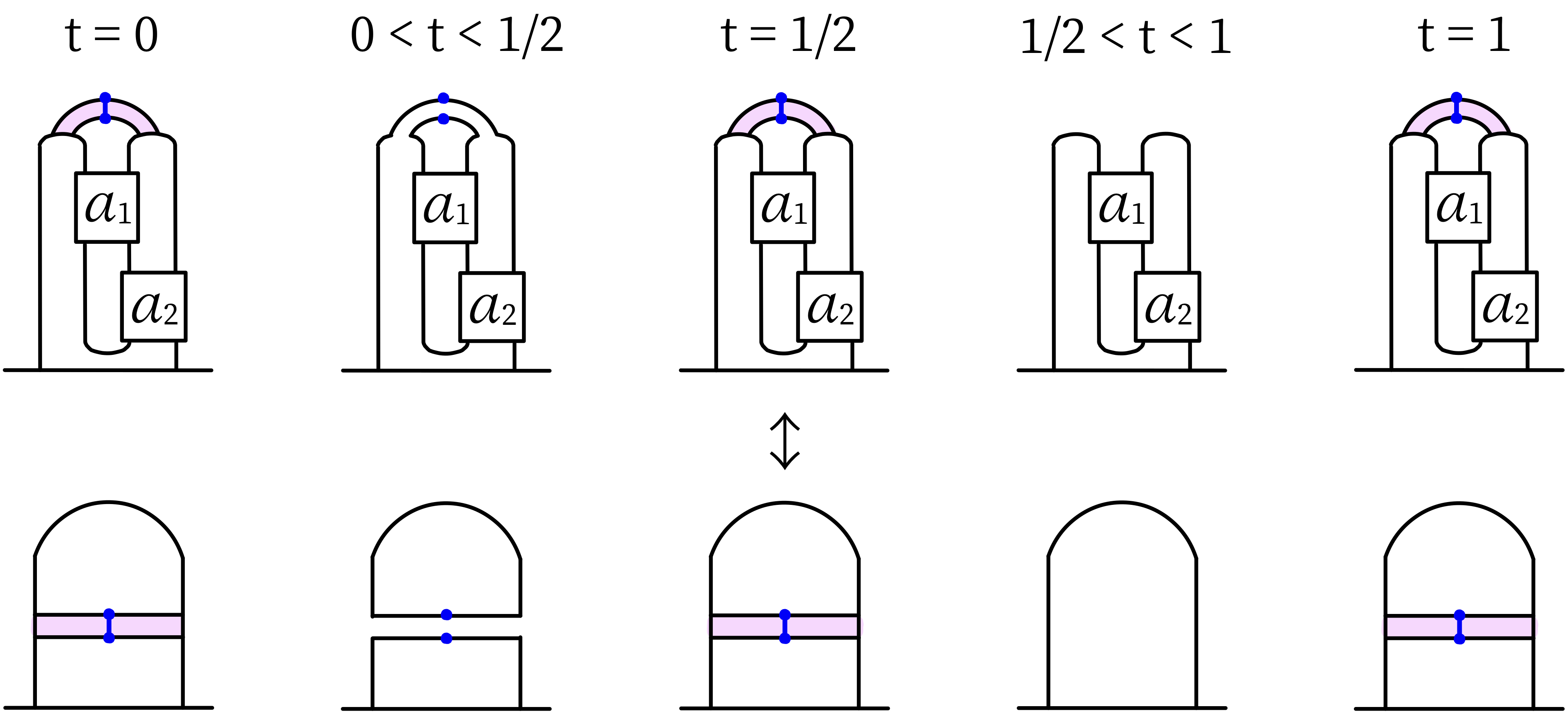}
\caption{Before and after deformations (i) and (ii).}
\label{2-bridge transform1}
\end{figure}

(iii) When $(a_1,a_2)\equiv(2,0) \pmod4$: 
Applying the operation (6) to the band makes it vertical and transfers $a_1$ twists onto it. Because $a_1\equiv2\pmod4$, eliminating the twists by the operation (3) leaves one full twist. This twist can be removed by the operation (4), during which the belt sphere winds once around the 1-handle. Applying the operation (6) again makes the band horizontal and transfers  $a_2$ twists onto it; since $a_2\equiv0\pmod4$, these twists can be eliminated by the operation (3). The resulting tangle is trivial, and the band is attached without twists.
 Figure \ref{2-bridge transform2} shows the motion pictures before and after the deformations (iii).
 The winding of the belt sphere around the 1-handle can be removed if the next index $i$ with $a_i\equiv2\pmod4$ is odd, by applying the operation (3) to reverse the full twist corresponding to $a_i$. 
 \begin{figure}[tbh]
\centering
\includegraphics[width=10cm]{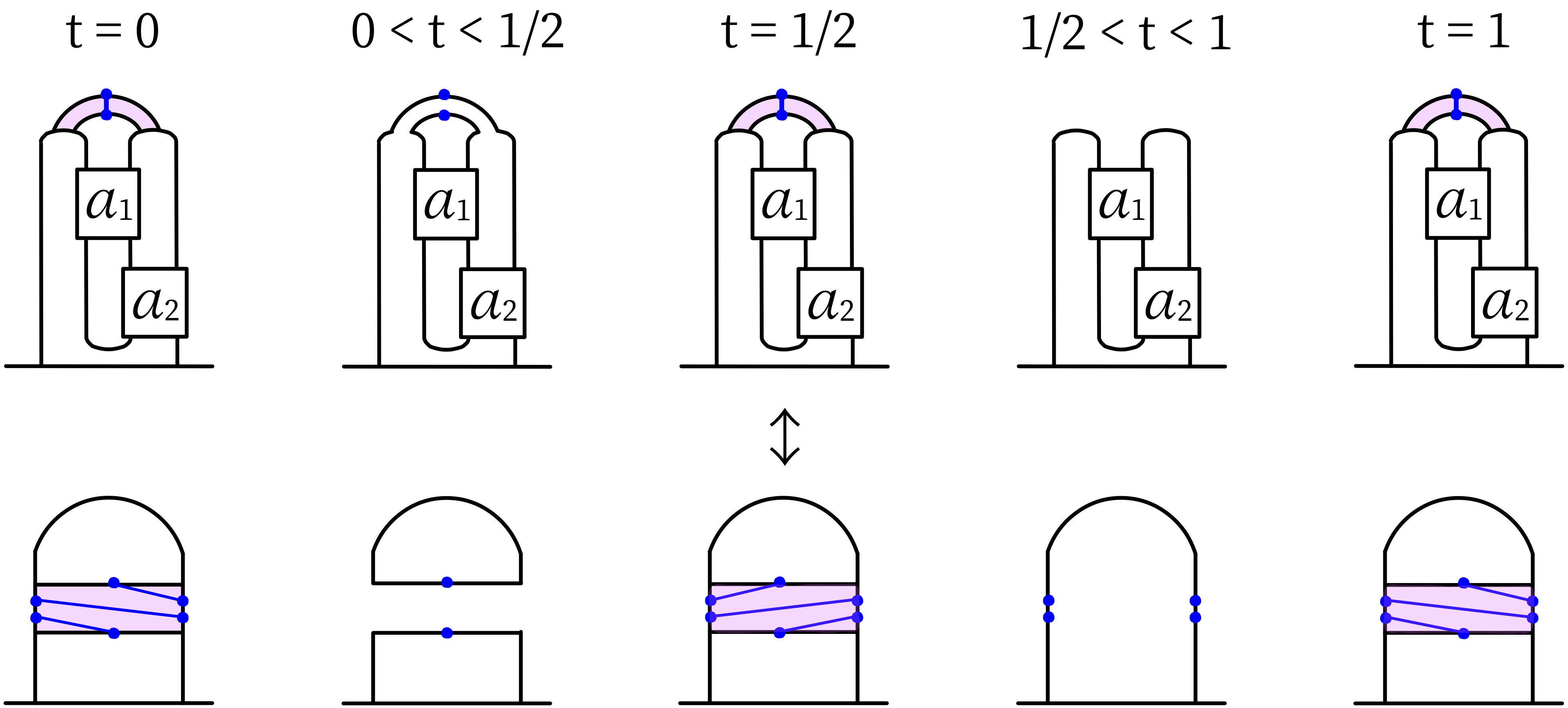}
\caption{Before and after deformation (iii).}
\label{2-bridge transform2}
\end{figure} 

(iv) When $(a_1,a_2)\equiv(2,2) \pmod4$: 
Applying the operation (6) to the band makes it vertical and transfers $a_1$ twists onto it. Since $a_1\equiv2\pmod4$, eliminating the twists by the operation (3) leaves one full twist. This twist can be removed by the operation (4), during which the belt sphere winds once around the 1-handle. Performing the operation (6) again makes the band horizontal, and $a_2$ twists are transferred onto it. As $a_2\equiv2\pmod4$, eliminating these twists by the operation (3) leaves one full twist. This twist can be removed by the operation (4). In trivializing the tangle, the belt sphere becomes wound once around the 1-handle in a direction perpendicular to that in the previous case. 
 Figure \ref{2-bridge transform3} shows the motion pictures before and after the deformations (iv).
\begin{figure}[tbh]
\centering
\includegraphics[width=10cm]{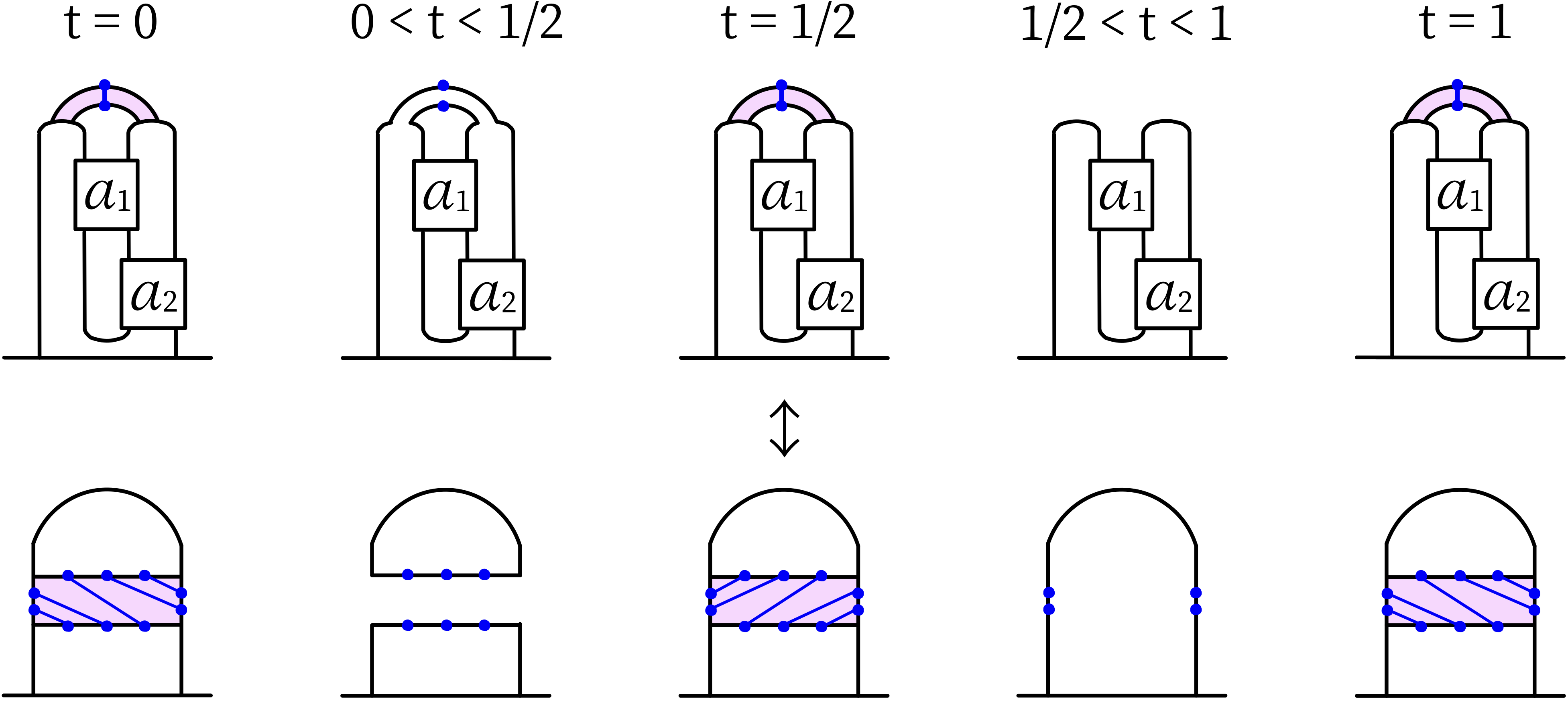}
\caption{Before and after deformation (iv).}
\label{2-bridge transform3}
\end{figure}

Since in cases (i), (ii), and (iii), the motion pictures after the deformations show that the belt sphere of the 1-handle corresponding to the band appears as a single point on both the upper and lower edges of the band, the conditions of Proposition 3.2 are satisfied.
However, when the belt sphere winds around the 1-handle as in the deformation (iv), the motion pictures after the deformation show that the belt sphere of the 1-handle corresponding to the band appears as three points on both the upper and lower edges of the band. In this situation, it is impossible to attach a 2-handle whose attaching sphere intersects the belt sphere transversely in a single point.
Therefore, the conditions of Proposition 3.2 are not satisfied in this form. 
The twist of the belt sphere created in resolving $a_2$ can be eliminated if the next $a_i\equiv2\pmod4$ occurs in an even position, by applying the operation (3) to reverse the full twist of $a_i$. Likewise, the twist of the belt sphere introduced when resolving 
$a_1$ can be removed—after cancelling the twist from 
$a_2$ —if the next $a_j\equiv2\pmod4$ occurs in an odd position, again by applying the operation (3) to reverse the corresponding full twist.

Since all of $a_1,a_2,\ldots,a_m$ are even, the sequence modulo $4$ consists only of $0$ and $2$. 
The terms equal to $0$ do not affect the belt sphere when eliminating crossings, so it suffices to consider the terms equal to $2$. 

Let $(b_1,b_2,\ldots,b_m)$ be the sequence obtained by reducing each $a_i$ modulo $4$. 
For each $b_i$, omit it if $b_i=0$, replace it with $O$ if $b_i=2$ with $i$ odd, and with $E$ if $b_i=2$ with $i$ even. This yields a word $w_k$ of length $k$ ($0 \leq k \leq m$) in the letters $O$ and $E$.

For the sequence $(b_1,b_2,\ldots,b_m)$, suppose that 
$b_1,b_2,\ldots,b_{i-1}$ are all equal to 0 and $b_i=2$ with $i$ even.
Then, using the deformation (i), we can eliminate $b_1,b_2,\ldots,b_{i-2}$. 
Moreover, $b_{i-1}$ and $b_i$ can be eliminated by using the deformation (ii).
Since the deformations (i) and (ii) do not introduce any twist in the belt sphere, $C(b_1,\ldots,b_m)$ can be deformed into $C(b_{i+1},\ldots,b_m)$.
This deformation corresponds to deleting the initial $E$ of the word $w_k$.

For the sequence $(b_1,b_2,\ldots,b_m)$, suppose there exist $i<j$ (both odd or both even) such that $b_i=b_j=2$ and $b_{i+1}=\dots=b_{j-1}=0$.
In this case, even if the belt sphere becomes twisted at $b_i$, the twist is cancelled when eliminating the crossing corresponding to $b_j$.
Thus, $C(b_1,\ldots,b_{i-1},b_i,0,\dots,0,b_j,b_{j+1},\dots,b_m)$ can be deformed into $C(b_1,\ldots,b_{i-1},0,b_{j+1},\dots,b_m)$.
This deformation corresponds to deleting the adjacent pair $OO$ or $EE$ in the word $w_k$.

The word $w_k$ can be shortened by repeatedly deleting the initial $E$ or deleting adjacent pairs $OO$ or $EE$.
If the word reduces to the empty word, $C(b_1,\ldots,b_m)$ can be deformed into $C(0,0)$ or $C(0,2)$. 
If the word reduces to $O$, $C(b_1,\ldots,b_m)$ can be deformed into $C(2,0)$. 
Therefore, if the word reduces to either the empty word or $O$, the corresponding 2-bridge knot tangle can be transformed into $(T_0,B_0)$. 
By Proposition~3.2, it follows that for any $n$, the $n$-twist spin of such a 2-bridge knot is 1-slice.

\end{proof}

\begin{cor}\label{2-bridgeCor}
    The slice depth of the 2-twist spin of a 2-bridge knot that satisfies the condition of Theorem 4.1 is equal to $1$.
\end{cor}

\begin{proof}
    Let $k=C(a_1,a_2,\ldots ,a_m)$ be a $2$-bridge knot represented by nonzero even integers 
$a_1,a_2,\ldots ,a_m$. 
We define an irreducible fraction $p/q$ with $0<|q|<p$ as
\begin{equation*} 
\frac{p}{q} = a_1+\cfrac{1}{a_2+ \cfrac{1}{\cdots +\cfrac{1}{a_m}}}. \end{equation*}
It is known that $p$ is odd and $q$ is even, and the determinant of $k$ is equal to $p$ (See \cite{Nanasawa} and \cite{KS}). 
In particular, we have $p>1$. 
By virtue of a result of Joseph \cite[ Corollary 1.4]{Joseph}, the $2$-twist spin $\tau^2(k)$ of $k$ is not $0$-slice. 
Consequently, the 2-twist spin of any 2-bridge knot is not 0-slice.
On the other hand, if $(a_1,a_2,\ldots ,a_m)$ satisfies the condition of Theorem \ref{THM2-bridge}, 
then $k$ is $1$-slice, and hence its slice depth $\operatorname{sd}(\tau^2(k))$ is equal to $1$. 
\end{proof}

\begin{exmp}
Let $k$, $p$, $q$ be defined as in Corollary \ref{2-bridgeCor}.
If $p$ and $q$ satisfy the following condition, then $k$ satisfies the condition of Theorem \ref{THM2-bridge}.

\noindent(i) $p^2-4pq-q^2=1$: The square root of five has the following infinite continued fraction expansion: $\sqrt{5}=[2,4,4,\dots]$.
Let $x_n/y_n$ denote the $n$-th convergent $[2,4,\dots,4]$, where the entry $4$ is repeated $n$ times. It is well known that for odd $n\in\mathbb{N}$, each pair $(x_n,y_n)$ is a solution to the Pell equation $x^2-5y^2=1$. Now, setting 
\begin{equation*}
    \frac{p_n}{q_n}=\frac{x_n}{y_n}+2,
\end{equation*}
we obtain
\begin{equation*}
    p_n^2-4p_nq_n-q_n^2=x_n^2-5y_n^2.
\end{equation*}
Hence, the finite continued fraction $p_n/q_n=[4,4,\dots,4]$ satisfies $p_n^2-4p_nq_n-q^2_n=1$.
Consequently, if $p$ and $q$ satisfy $p^2-4pq-q^2=1$, then $k$ is the 2-bridge knot $C(4,\dots,4)$. In this case, the word $w_k$ is empty, and thus $k$ satisfies the condition of Theorem \ref{THM2-bridge}.
For instance, the first few values are
\begin{align*}
    &[4,4]=\frac{p_1}{q_1}=\frac{17}{4},\\
    &[4,4,4,4]=\frac{p_3}{q_3}=\frac{305}{72},\\
    &[4,4,4,4,4,4]=\frac{p_5}{q_5}=\frac{5473}{1292}.\\
\end{align*}

(ii) $p^2-4pq-2q^2=1$: The square root of six has the following infinite continued fraction expansion: $\sqrt{6}=[2,2,4,2,4,\dots]$.
Let $x_n/y_n$ denote the $n$-th convergent $[a_0,a_1,a_2,\dots a_n]$. It is well known that for odd $n\in\mathbb{N}$, each pair $(x_n,y_n)$ is a solution to the Pell equation $x^2-6y^2=1$. Now, setting 
\begin{equation*}
    \frac{p_n}{q_n}=\frac{x_n}{y_n}+2,
\end{equation*}
we obtain
\begin{equation*}
    p_n^2-4p_nq_n-2q_n^2=x_n^2-6y_n^2.
\end{equation*}
Hence, the finite continued fraction $p_n/q_n=[4,2,\dots,4,2]$ satisfies $p_n^2-4p_nq_n-2q_n^2=1$.
Consequently, if $p$ and $q$ satisfy $p^2-4pq-2q^2=1$, then $k$ is the 2-bridge knot $C(4,2,\dots,4,2)$. In this case, the word $w_k$ is $E\dots E$, and thus $k$ satisfies the condition of Theorem \ref{THM2-bridge}.
For instance, the first few values are

\begin{align*}
    &[4,2]=\frac{p_1}{q_1}=\frac{9}{2},\\
    &[4,2,4,2]=\frac{p_3}{q_3}=\frac{89}{20},\\
    &[4,2,4,2,4,2]=\frac{p_5}{q_5}=\frac{881}{198}.\\
\end{align*}

\end{exmp}


\subsection{Pretzel knots}

A knot given by the diagram in Figure \ref{pretzel1} corresponding to a sequence of nonzero integers $(a_1,\dots, a_m)$ is called a pretzel knot, denoted by $P(a_1,\dots, a_m)$.
In the diagram, $a_i$ denotes $a_i$ half-twists.
For details on pretzel knots, see Kawauchi \cite{Kawauchi}.

\begin{figure}[tbh]
\centering
\includegraphics[width=5cm]{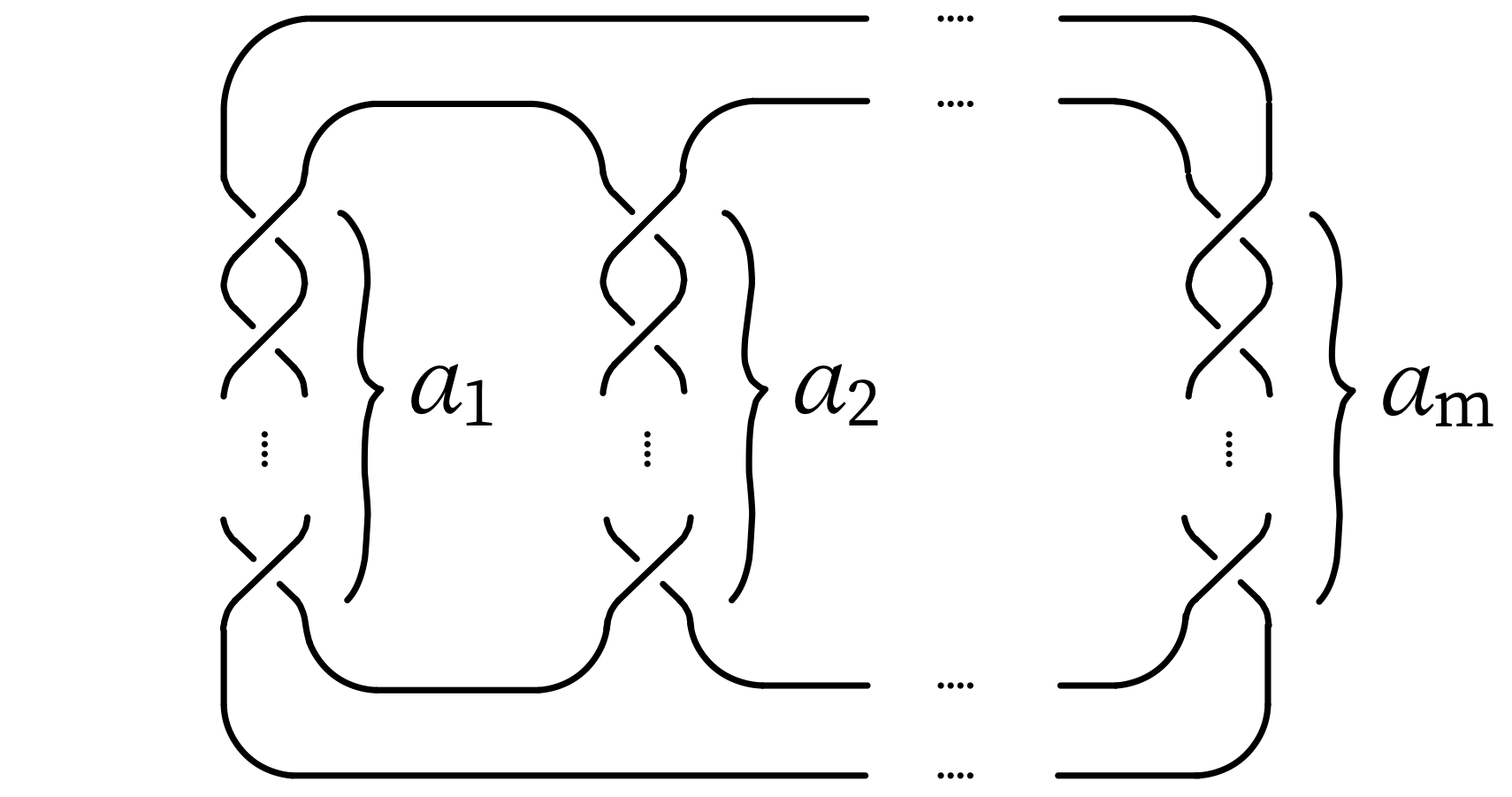}
\caption{Diagram of a pretzel knot $P(a_1,\dots, a_m)$.}
\label{pretzel1}
\end{figure}

\begin{thm}\label{THMPretzel}
    For any positive integer $n$, the $n$-twist spin of the pretzel knot $P(4i+1, 8i+1, 8i+3)$ is 2-slice for any $i\in\mathbb{N}$.
\end{thm}

\begin{proof}

\begin{figure}[tbh]
\centering
\includegraphics[width=10cm]{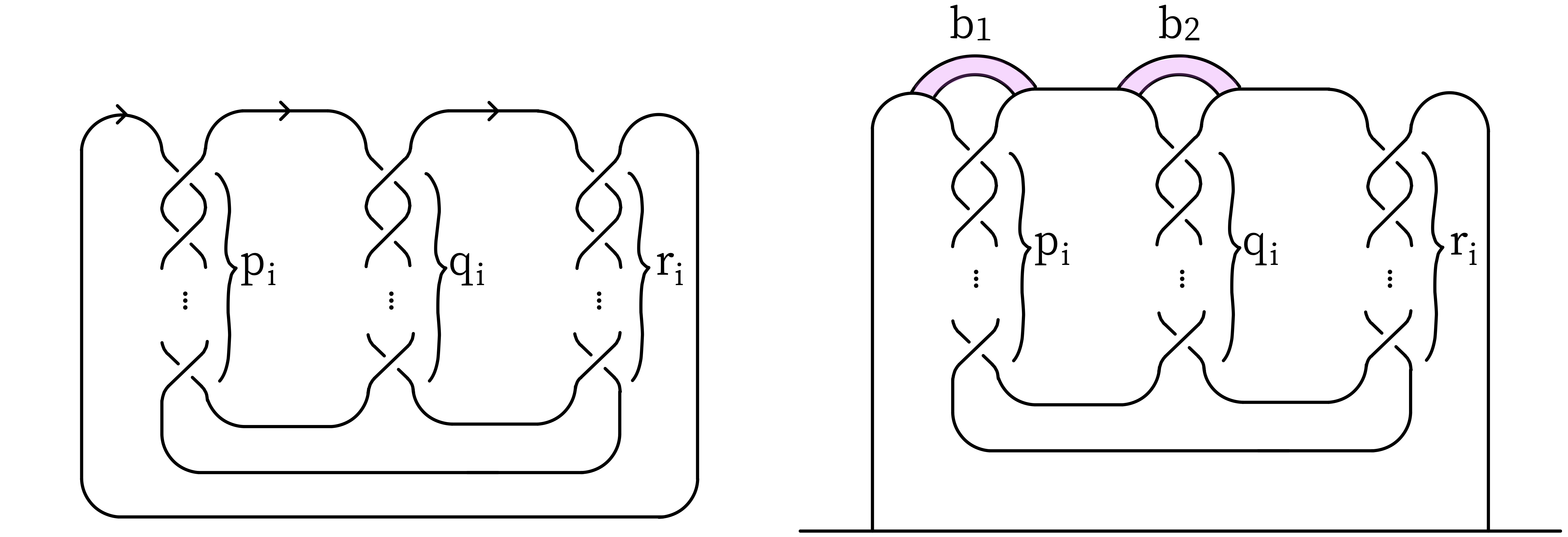}
\caption{$P(p_i, q_i, r_i)$ and the corresponding banded tangle.}
\label{pretzel2}
\end{figure}

  We put $p_i=4i+1, q_i=8i+1, r_i=8i+3$. We attach two bands $b_1$ and $b_2$ to the tangle of the pretzel knot as illustrated in the right diagram of Figure \ref{pretzel2}. By applying the operation (6) to each band, we transfer $p_i$ and $q_i$ half twists to $b_1$ and $b_2$, respectively. Subsequently, by removing the $r_i$ half-twists, the total number of half-twists accumulated in $b_1$ and $b_2$ becomes $p_i+r_i=12i+4$ and $q_i+r_i=16i+4$, respectively. The two bands also become entangled with $r_i$ crossings between them.
  Since $r_i=4i+3$ is odd, we can reduce the number of crossings between $b_1$ and $b_2$ to one by appropriately switching the over- and under-crossings using the operation (3). Moreover, as the accumulated half-twists in both bands are divisible by 4, all twists can be eliminated via repeated application of the operation (3) (See Remark 3.1).
  
  We then apply the operation (6) once again. By deforming the tangle into a trivial one, the diagram transforms into the configuration shown in Figure \ref{pretzel deformation}.
  After repositioning the ends of the bands using the operation (5), the resulting diagram satisfies the conditions of $(T_0, \mathcal{B}_0)$ described in Section 3. 
 Therefore, since the tangle of the pretzel knot $P(4i+1, 8i+1, 8i+3)$ can be transformed into $(T_0, \mathcal{B}_0)$, it follows from Proposition 3.2 that the $n$-twist spin of the given pretzel knot is 2-slice for any $n$.

\begin{figure}[tbh]
\centering
\includegraphics[width=10cm]{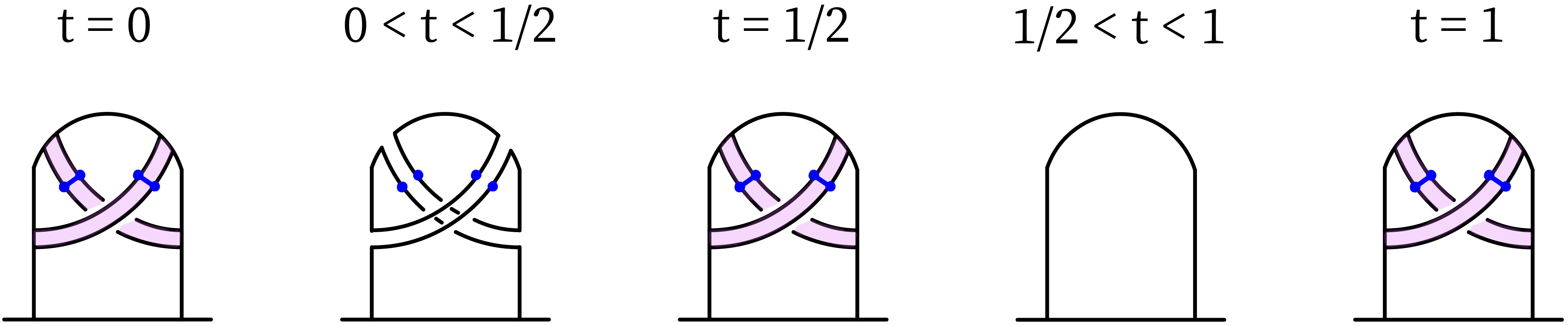}
\caption{Motion picture after the second (6).}
\label{pretzel deformation}
\end{figure}

\end{proof}

\begin{cor}\label{}
    The slice depth of the $2$-twist spin of the pretzel knot $P(4i+1, 8i+1, 8i+3)$ is equal to $1$ or $2$.
\end{cor}

\begin{proof}
    It is regarded as monoid that the quotient set of $2$-knots in $S^4$ by $0$-concordance and the operation of connected sum. This is called $0$-concordance monoid denoted by $\mathcal{M}_0$.
    Based on the result of Dai and Miller \cite{DM}, the family 
    \begin{align*}
        \mathcal{F}=\{\text{2-twist spin of }P(2i+1, 4i+1,4i+3)\},
    \end{align*}
     is linearly independent in $\mathcal{M}_0$; that is, none of $P(2i+1, 4i+1,4i+3)$ are 0-slice.
Therefore, in combination with Theorem 4.4, the slice depth of the 2-twist spin of the pretzel knot $P(4i+1, 8i+1, 8i+3)$ is equal to either 1 or 2.
\end{proof}


\subsection{Ribbon knots of $1$-fusion}

A ribbon knot of 1-fusion is a band sum of 2-component trivial link and its diagram is illustrated on the left of Figure \ref{ribbon knot of 1-fusion}. The integers $a_1,a_2,\dots,a_m$ specify the number of times the band winds around one of the trivial components as shown on the right of Figure~\ref{ribbon knot of 1-fusion}, with negative values corresponding to windings in the opposite direction. In the central rectangle, a part of the single band appears as $m+1$ subbands, which may exhibit twists or have portions that pass over or under each other. When considering two distinct subbands, they may cross over and under each other or become entangled. The left and right ends of the subband are connected at the same vertical level. Due to the coherent orientation of the two trivial knots, the ribbon band must be either untwisted or twisted by an even number of half-twists. We denote such a ribbon knot of 1-fusion by $R(a_1,\ldots,a_m)$.
For details on ribbon knots of 1-fusion, see Mizuma \cite{Mizuma}.

\begin{thm}\label{THMribbon}
    For any positive integer n, the n-twist spin of the ribbon knot of 1-fusion $R(a_1, \dots ,a_m)$ is 2-slice if the following condition holds:
    \begin{eqnarray*}
        a_1+a_2+\dots+a_m+\sigma+w\equiv0 \mod2, 
    \end{eqnarray*}
    where
    \begin{itemize}
        \item $\sigma$ is the signed sum of the full twists originally present in the ribbon band, and
        \item $w$ is the sum of the crossing signs of the core of the ribbon band, counted after modifying the crossings so that the band becomes unknotted.
    \end{itemize}
\end{thm}

\begin{proof}

     \begin{figure}[tbh]
\centering
\includegraphics[width=9cm]{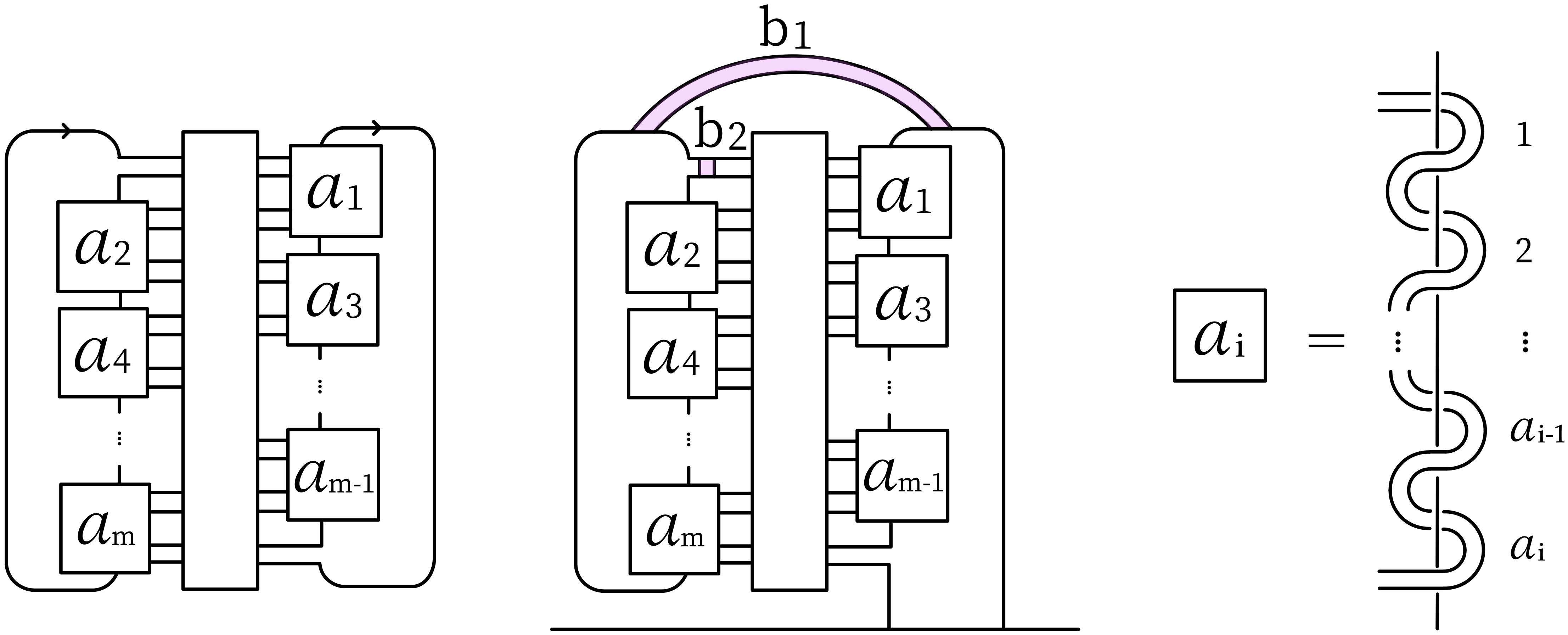}
\caption{$R(a_1, \dots ,a_m)$ and the corresponding banded tangle.}
\label{ribbon knot of 1-fusion}
\end{figure}

    As illustrated on the middle of Figure \ref{ribbon knot of 1-fusion}, we attach two band $b_1$ and $b_2$ to the tangle of the ribbon knot.
We then apply the operation (6) to the band $b_2$. As a result, $b_2$ takes the place of the original ribbon band. By suitably applying the operation (3) to adjust the over/under crossings, the entanglement of the band can be eliminated.
 Consequently, $m+1$ parallel subbands appear inside the rectangle. At this point in the  process,the band accumulates $\sigma+w$ full twists.
Next, we slide $b_2$ over $b_1$ to eliminate the $a_1$ windings. During this process,  $b_2$ acquires $a_1$ full twists. We then slide $b_2$ over $b_1$ again to eliminate the $a_2$ windings, resulting in $b_2$ acquiring an additional $a_2$ full twists. By repeating this process to eliminate all the windings corresponding to $a_1, a_2, …, a_m$, $b_2$ accumulates a total of $a_1+a_2+\dots+a_m+\sigma+w$ full twists.
 Since $a_1+a_2+\dots+a_m+\sigma+w$ is even, then all full twists can be eliminated by applying the operation (3) repeatedly (See Remark 3.1). Subsequently, applying the operation (6) to the band $b_2$ and sliding it into an appropriate position transforms the diagram into $(T_0, \mathcal{B}_0)$ described in Section 3. 
 Therefore, since the tangle of the ribbon knot of $1$-fusion $R(a_1, …a_m)$ that satisfies $a_1+a_2+\dots+a_m+\sigma+w$ is even can be transformed into $(T_0, \mathcal{B}_0)$, it follows from Proposition 3.2 that the $n$-twist spin of the given ribbon knot is 2-slice for any $n$.

\end{proof}


\subsection{Classical knots with unknotting number $u$}
The unknotting number of a classical knot is the minimal number of crossing changes required to transform the knot into the unknot.
According to Yamamoto \cite[Lemma 2.1]{Yamamoto}, a classical knot with unknotting number $u$ admits a diagram of the following form. Let $\gamma_0$ be a trivial circle, and let $\gamma_1, \dots, \gamma_u$ be trivial circles each linked once with $\gamma_0$. For each $i=1, \dots, u$, let $S_i$ be a band connecting $\gamma_0$ and $\gamma_i$. These satisfy the following conditions:
\begin{enumerate}
    \item $\gamma_0 \cap S_i=\gamma_0 \cap \partial S_i=\beta _i$, 
    \item $\gamma_i$ meets $S_i$ in an arc $\delta _i \subset \partial S_i$, 
    \item $\gamma_i$ does not intersect $S_j$ for $i\neq j$, 
    \item $\partial S_i$ winds none of $\gamma_0, \dots, \gamma_u$ except for $\alpha _i$ in the diagram, 
    \item the diagram $D(k)$ of the knot $k$ is given by \\$D(k)=(\gamma _0\cup \dots \cup\gamma _u\cup\partial S_1\cup \dots \cup\partial S_u)-\text{Int} (\beta_1\cup \dots \cup\beta _u\cup\delta_1\cup \dots \cup\delta_u)$, 
    \end{enumerate}
where $i,j=1, \dots, u$ (See Figure \ref{u}).

\begin{figure}[tbh]
\centering
\includegraphics[width=12.6cm]{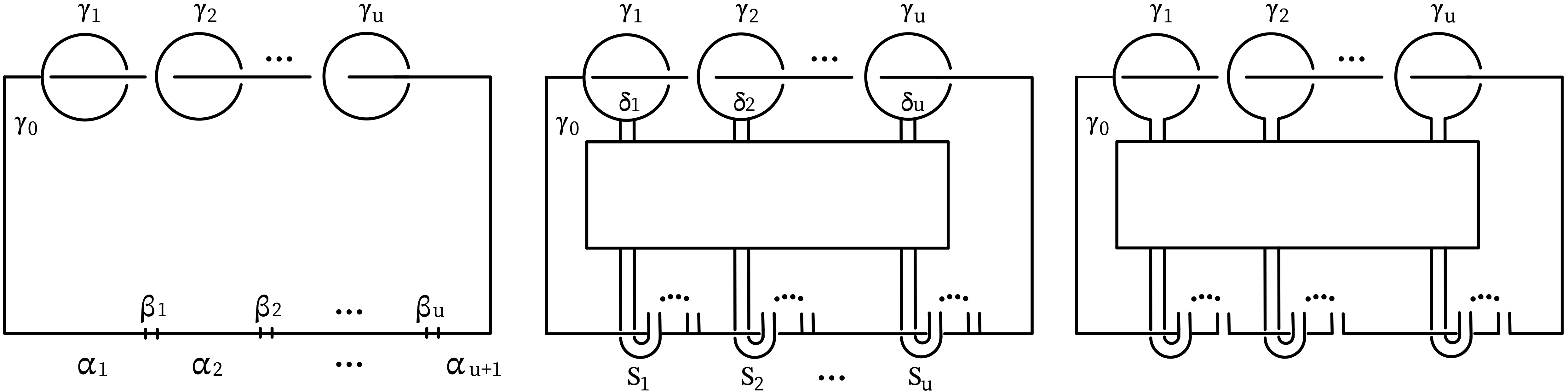}
\caption{Construction of $D(k)$  for $k$ with unknotting number $u$.}
\label{u}
\end{figure}

If the band $S_i$ contains an odd number of half-twists, we can modify the diagram by flipping the circle $\gamma_i$, thereby making the total number of half-twists even.

In what follows, we work with a diagram that satisfies the above conditions.

\begin{thm}\label{THMunknotting}
    For any positive integer n, the n-twist spin of a classical knot with unknotting number $u$ is $u$-slice if the following condition holds:
    \begin{eqnarray*}
        \sigma_i+w_i+\lambda_i\equiv0 \mod2 \quad\text{for all} \quad i=1, \dots ,u 
    \end{eqnarray*}
    where
    \begin{itemize}
        \item $\sigma_i$ denotes the signed sum of full twists originally present in the band $S_i$, 
        \item $w_i$ is the sum of the crossing signs of the core of $S_i$, counted after modifying the crossings so that the band becomes unknotted, 
        \item $\lambda_i$ is the number of times $S_i$ wraps around  $\gamma_0$.

    \end{itemize}
    See Figure \ref{unknotting number with band}.
\end{thm}

    \begin{figure}[tbh]
\centering
\includegraphics[width=11cm]{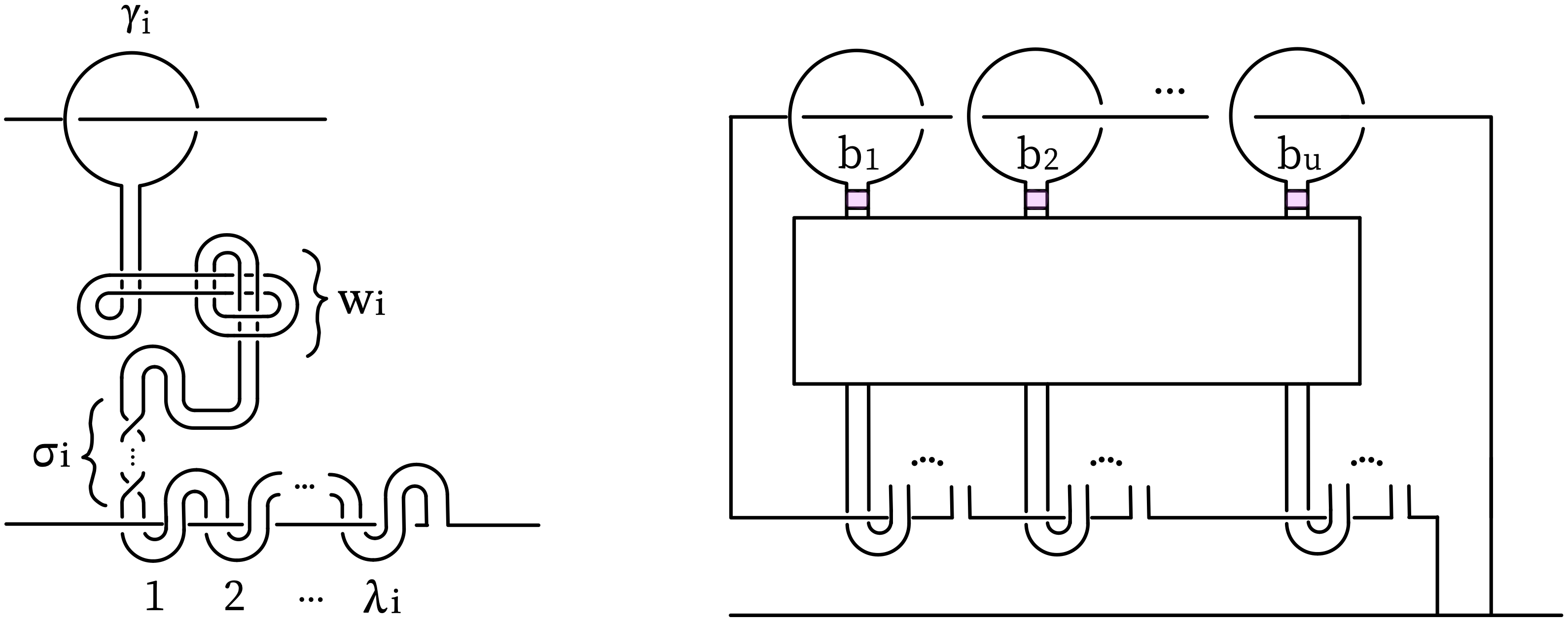}
\caption{Left: $\sigma_i$, $w_i$, $\lambda_i$ indicated. Right: the corresponding banded tangle of a classical knot with unknotting number $u$.}
\label{unknotting number with band}
\end{figure}

\begin{proof}
    To the tangle in the diagram satisfying conditions (1) through (5) above, we attach a band $b_i$ to each $S_i$ as illustrated on the right in Figure \ref{unknotting number with band}. In the central rectangle, $S_i$ may exhibit twists or have parts that pass over or under each other. When considering two distinct bands $S_i$ and $S_j$, they may cross over and under each other or become entangled.
    We then apply the operation (6) to every $b_i$. By appropriately applying the operation (3), we can adjust the over/under crossings to completely separate the bands $b_i$ from each other. If a band $b_i$ is knotted, we again apply the operation (3) to modify the crossings and untie it. We then modify each band $b_i$ so that it descends vertically from $\delta_i$, winds around $\gamma_0$ exactly $\lambda_i$ times, and attaches to $\beta_i$. During this process, $w_i$ full twists are added to $b_i$.

    Next, we remove the $\lambda_i$ windings by sliding $b_i$ along $\gamma_i$, which results in an additional $\lambda_i$ full twists being introduced into $b_i$. Consequently, the band $b_i$ becomes a straight band connecting  $\beta_i$ and $\delta_i$, and the total number of full twists it contains is $ \sigma_i+w_i+\lambda_i$.
    Since $\sigma_i+w_i+\lambda_i$ is even, all full twists can be eliminated by applying the operation (3).
    We then apply the operation (6) to $b_i$ once again. Since the circle $\gamma_1$ now corresponds to a single full twist as shown in Figure \ref{unknotting number transformation}, we transfer this twist to the band $b_1$, and eliminate it using the operation (4). We perform the same procedure for $\gamma_2, \gamma_3, \dots, \gamma_u$ sequentially.

    As a result, the diagram reduces to one in which the trivial tangle is attached by  $u$ mutually disjoint, unknotted, untwisted bands. By performing appropriate band slides, we obtain the diagram $(T_0, \mathcal{B}_0)$ described in Section 3. 
    Therefore, since the tangle of classical knot with unknotting number $u$ that satisfies the condition in Theorem~\ref{THMunknotting} can be transformed into $(T_0, \mathcal{B}_0)$, it follows from Proposition~\ref{sliceprop} that the $n$-twist spin of the given classical knot is $u$-slice for any $n$.

    \begin{figure}[tbh]
\centering
\includegraphics[width=12.6cm]{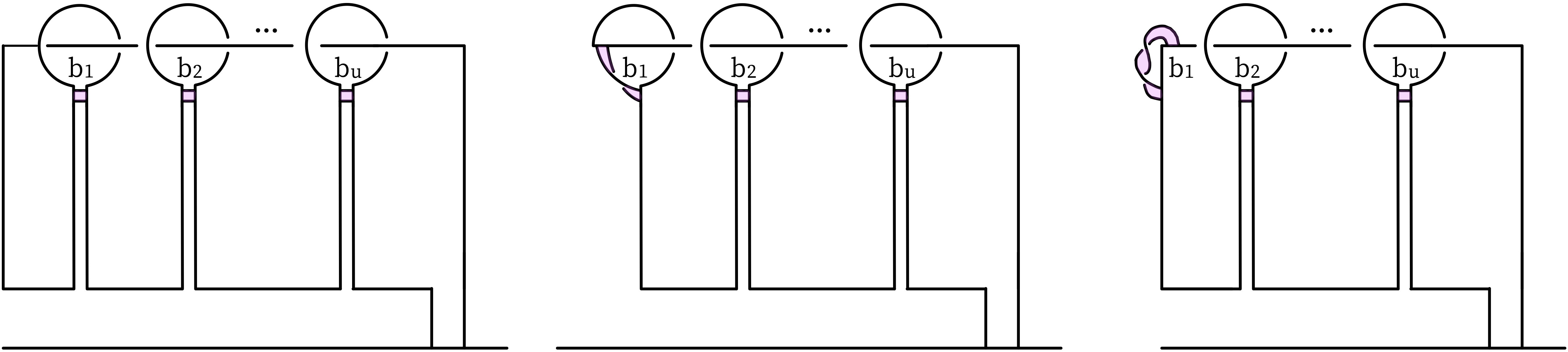}
\caption{Process of transferring the full twist from $\gamma_1$ to $b_1$.}
\label{unknotting number transformation}
\end{figure}

\end{proof} 


\section{Observations and Application}
\label{Observations and Application}


In this section, we first compare the slice depth with the unknotting number for twist spun 2-knots. We then extract examples from the Rolfsen knot table for which the slice depth can be explicitly determined using our result.

In our method, we have shown that if a pair $(T,\mathcal{B})$ obtained by attaching $m$ bands to a tangle representing a knot $k$ can be transformed into $(T_0,\mathcal{B}_0)$, then the $n$-twist spin of $k$ is $m$-slice. In the context of Satoh's work, however, the same situation implies that the $n$-twist spin of $k$ has unknotting number $m$.
For example, our result shows that the $n$-twist spin of a $2$-bridge knot is $1$-slice, while Satoh's result implies that its unknotting number is equal to 1.
This naturally leads to the question: can the slice depth and the unknotting number differ?

One essential difference is that the slice depth can be zero even when the 2-knot is not trivial.
For instance, every ribbon 2-knot is 0-slice, meaning that its slice depth is 0, while the unknotting number of any nontrivial knot is always at least 1.
As a concrete example, the spun 2-knot of the trefoil is 0-slice, but its unknotting number is 1 \cite{HMS}.

Although our method imposes stricter conditions on the deformation of 1-handles, and one might conjecture that there exist 2-knots whose slice depth is greater than their unknotting number, we were unable to determine whether such examples exist.

We now present examples in which the slice depth can be determined by our results. According to Joseph \cite[Theorem1.3]{Joseph}, let $k$ be any classical knot with Alexander polynomial $\Delta_k(t)$ such that $|\Delta_k(-1)|\neq1$, then there exists an integer $n$ for which the $n$-twist spin of $k$ is not 0-slice.
In Rolfsen knot table, all classical knots with crossing number at most 10 except for $10_{124}$ and $10_{153}$, satisfy $|\Delta_k(-1)|\neq1$.
Among these, there are 95 2-bridge knots. 

Among the 2-bridge knots in the Rolfsen knot table with crossing number at most 10, the following 29 knots satisfy the conditions of Theorem~\ref{THM2-bridge}:
\begin{align*}
&5_2, 6_1, 7_4, 7_5, 8_3, 8_6, 8_8, 8_{14}, 9_2, 9_5, 9_8, 9_9, 9_{10}, 9_{15},9_{18}, 9_{19}, 10_1, 10_{3}, \\
&10_{12}, 10_{14}, 10_{16}, 10_{18}, 10_{22}, 10_{24},10_{25}, 10_{27}, 10_{31}, 10_{33}, \text{and } 10_{35}.
\end{align*}
Hence the slice depth of the $2$-twist spin of each of these knots is equal to $1$ by Corollary \ref{2-bridgeCor}.

\section*{acknowledgement}
The author would like to express sincere gratitude to Professor Hisaaki Endo for his invaluable advice and continuous support throughout the entire course of this research, from its early stages to the completion of this paper.

\end{document}